\documentclass{amsart}

\AtEndDocument{\bigskip{\footnotesize
  \textsc{Instituto Nacional de Matem\'atica Pura e Aplicada, Rio de Janeiro, Brazil} \par  
  \textit{E-mail}:  \texttt{vgbramos@impa.br}
  }}

\usepackage{amssymb}
\usepackage{latexsym}
\usepackage{amsmath}
\usepackage{bm}
\usepackage{subcaption}
\usepackage{xcolor}
\usepackage[all]{xy}
\usepackage{overpic,graphicx}
\usepackage[margin=1.5in,footskip=0.4in]{geometry}

\newtheorem{theorem}{Theorem}[section]
\newtheorem{proposition}[theorem]{Proposition}

\newtheorem{lemma}[theorem]{Lemma}

\newcommand{\R}{{\mathbb R}}

\newcommand{\Z}{{\mathbb Z}}

\renewcommand{\epsilon}{\varepsilon}

\newcommand{\ba}{\mathbf{a}}
\newcommand{\bx}{\mathbf{x}}
\newcommand{\by}{\mathbf{y}}
\renewcommand{\aa}{\bm{\alpha}}
\newcommand{\bb}{\bm{\beta}}
\newcommand{\Vect}{\text{Vect}}
\newcommand{\gr}{\text{\textnormal{gr}}}

\title{Absolute gradings on ECH and Heegaard Floer homology}

\author[Vinicius G. B. Ramos]{Vinicius Gripp Barros Ramos}

\begin{document}
\begin{abstract}
In joint work with Yang Huang, we defined a canonical absolute grading on Heegaard Floer homology by homotopy classes of oriented 2-plane fields. A similar grading was defined on embedded contact homology by Michael Hutchings. In this paper we show that the isomorphism between these homology theories defined by Colin-Ghiggini-Honda preserves this grading.
\end{abstract}
\maketitle

\section{Introduction}

Let $Y$ be a closed, connected and oriented three-manifold. The embedded contact homology (ECH) and the Heegaard Floer homology of $Y$ are invariants that have been studied and computed for many manifolds. ECH was defined by Hutchings using a contact form on $Y$, see~\cite{hutcast}, and Heegaard Floer homology was defined in~\cite{OS1} by Ozsv\'{a}th-Szab\'{o} using a Heegaard decomposition of $Y$. These two homology theories have very distinct flavors, but they have recently been shown to be isomorphic by Colin-Ghiggini-Honda~\cite{cgh1,cgh2,cgh3}. More specifically, they construct an isomorphism $\Phi:HF^+(-Y)\to ECH(Y)$. Here $HF^+(-Y)$ is a version of Heegaard Floer homology of $Y$ with the opposite orientation, which is isomorphic to the $-$ version of Heegaard Floer cohomology of $Y$ with its original orientation.

In joint work with Yang Huang~\cite{hr}, we defined a canonical absolute grading on Heegaard Floer homology by homotopy classes of oriented 2-plane fields on $Y$. A similar absolute grading had been defined in ECH by Hutchings \cite{ir}. Since these absolute gradings are defined in very different ways, it is not obvious that the isomorphism $\Phi$ would preserve them. On the other hand, when a contact form is given, it follows from Colin, Ghiggini and Honda's work that $\Phi$ maps what is called the contact invariant in one Floer homology to that of the other. It follows that in that particular spin-c structure, the absolute grading is preserved. The goal of this paper is to show that this holds in all generality as we now explain.

The orientation of $Y$ induces an isomorphism from this set to the set of homotopy classes of nonvanishing vector fields $\Vect(Y)$. So in this paper we will do all of our constructions with $\Vect(Y)$. For $\rho\in\Vect(Y)$, let $HF^+_{\rho}(-Y)$ and $ECH_{\rho}(Y)$ denote the submodules of $HF^+(-Y)$ and $ECH(Y)$, respectively, consisting of all elements of grading $\rho\in\Vect(Y)$. The main result of this paper is the following theorem.
\begin{theorem}\label{mainthm}
Let $\Phi:HF^+(-Y)\to ECH(Y)$ be the isomorphism constructed by Colin-Ghiggini-Honda. Then $\Phi$ maps $HF^+_{\rho}(-Y)$ to $ECH_{\rho}(Y)$ for all $\rho\in\Vect(Y)$.
\end{theorem}

We recall that both $HF^+(-Y)$ and $ECH(Y)$ admit a map $U$ whose mapping cone is denoted by $\widehat{HF}(-Y)$ and $\widehat{ECH}(Y)$, respectively. In order to show that $\Phi$ is an isomorphism, Colin, Ghiggini and Honda first construct an isomorphism $\widehat{\Phi}:\widehat{HF}(-Y)\to \widehat{ECH}(Y)$. They also show that the following diagram commutes.
\begin{equation}
\xymatrix{
\widehat{HF}(-Y) \ar[d]^{\widehat{\Phi}} \ar[r]^{\iota_*}& HF(-Y) \ar[d]^{\Phi}\\
\widehat{ECH}(Y)  \ar[r]^{\iota_*} & ECH(Y)
}\label{eq:comm}
\end{equation}

Here the horizontal maps $\iota_*$ denote the natural maps given by the mapping cone construction. In order to show that $\Phi$ preserves the absolute grading, it is enough to prove that both maps $\iota_*$ and $\widehat{\Phi}$ do.

The map $\widehat{\Phi}$ is defined as a composition $\widehat{\Phi}=\psi\circ\widetilde{\Phi}\circ\psi'$ as follows.
\begin{equation*}
\widehat{HF}(-Y)\stackrel{\psi'}{\longrightarrow}\widehat{HF}(S,\ba,\varphi(\ba))\stackrel{\widetilde{\Phi}}{\longrightarrow} ECH_{2g}(N_{(S,\varphi)},\lambda)\stackrel{\psi}{\longrightarrow}\widehat{ECH}(Y).
\end{equation*}
Here $\widehat{HF}(S,\ba,\varphi(\ba))$ is the homology of a chain group computed from the page of an open book decomposition $(S,\varphi)$ of $Y$ and $ECH_{2g}(N_{(S,\varphi)},\lambda)$ is the homology of a chain complex of generated by sets of Reeb orbits whose total intersection with a page is $2g$, where $\lambda$ is an appropriate contact form on $Y$. The maps $\psi'$, $\widetilde{\Phi}$ and $\psi$ are all isomorphisms and we will show that all of them preserve the absolute grading.

This paper is organized as follows. In Section~\ref{sec:abs}, we review the definitions of chain complexes of Heegaard Floer homology and ECH and the absolute grading on them. We explain how the chain complexes $\widehat{CF}(S,\ba,\varphi(\ba))$ and $ECC_{2g}(N_{(S,\varphi)},\lambda)$ are obtained from an open book decomposition and how the absolute grading is defined on them. We also show that $\psi'$ preserves the grading. In Section~\ref{sec:thm}, we recall some of the steps to construct the isomorphism $\widetilde{\Phi}$ and we prove that it preserves the absolute grading. This is the core of the proof of Theorem \ref{mainthm}. Finally, in Section \ref{sec:det}, we recall the construction of the map $\psi$ and we prove that it preserves the grading, finishing the proof of Theorem \ref{mainthm}. 

\section{The absolute gradings}\label{sec:abs}

\subsection{The grading on Heegaard Floer homology}\label{sec:absgrad}

A pointed Heegaard diagram is a quadruple $(\Sigma,\aa,\bb,z)$, where $\Sigma$ is a closed oriented surface of genus $g$, the tuples $\aa=(\alpha_1,\dots,\alpha_g)$ and $\bb=(\beta_1,\dots,\beta_g)$ are $g$-tuples of disjoint circles on $\Sigma$ which are linearly independent in $H_1(\Sigma)$ and $z$ is a point on $\Sigma$ in the complement of all of the circles $\alpha_i$ and $\beta_j$. Given a pointed Heegaard diagram $(\Sigma,\aa,\bb,z)$, an intersection point is a $g$-tuple $\bx=(x_1,\dots,x_g)$, where $x_i\in\alpha_i\cap\beta_{\sigma(i)}$ and $\sigma$ is a permutation of $\{1,\dots,g\}$. The chain complex $\widehat{CF}(\Sigma,\aa,\bb,z)$ is the $\Z$-module generated by the intersection points. One can define a differential $\partial$ on this complex and one can prove that $\partial^2=0$. The homology of this chain complex is denoted by $\widehat{HF}(Y)$. It can be shown that the homology does not depend on the pointed Heegaard diagram and hence it is an invariant of $Y$. For details, see~\cite{OS1}.

We now recall the definition of the other versions of Heegaard Floer homology. The complex $CF^{\infty}(\Sigma,\aa,\bb,z)$ is defined to be the $\Z$-module generated by $[\bx,n]$, where $\bx$ is an intersection point and $n\in\Z$. One can extend $\partial$ to $CF^{\infty}(\Sigma,\aa,\bb,z)$ so that $\partial^2=0$.
One can now define $CF^-(\Sigma,\aa,\bb,z)$ to be the submodule of $CF^{\infty}(\Sigma,\aa,\bb,z)$ generated by $[\bx,n]$, for $n<0$. One also defines $CF^+(\Sigma,\aa,\bb,z)$ to be the quotient of $CF^{\infty}(\Sigma,\aa,\bb,z)$ by $CF^-(\Sigma,\aa,\bb,z)$. The homologies of these complexes are denoted by
$HF^{\infty}(Y),\,HF^-(Y)\text{ and }HF^+(Y)$, respectively.

We will now recall the absolute grading on these homology groups. Let $(f,V)$ be a pair consisting of a self-indexing Morse function $f$ on $Y$ and a gradient-like vector field $V$, i.e. $df(V)>0$, whenever $df\neq 0$. We assume that $f$ has exactly one index 0 and one index 3 critical points. We also assume that all stable and unstable manifolds intersect transversely. For each index 1 critical point $p_i$, let $U_i$ denote the unstable manifold containing $p_i$ and, for each index 2 critical point $q_j$, let $S_j$ denote the stable manifold containing $q_j$. The pair $(f,V)$ is said to be compatible with the Heegaard diagram $(\Sigma,\aa,\bb)$ if
\begin{itemize}
 \item $\Sigma=f^{-1}(3/2)$,
 \item $\alpha_i=U_i\cap\Sigma$ and $\beta_j=S_j\cap\Sigma$, for all $1\le i,j\le g$.
\end{itemize}

An intersection point $\bx$ determines $g$ flow lines $\gamma_1,\dots,\gamma_g$ connecting the points $p_i$ to the points $q_j$. The basepoint $z$ determines a flow line $\gamma_0$ from the index 0 critical point to the index 3 critical point. Outside the union of small neighborhoods of $\gamma_0,\dots,\gamma_g$, which we denote by $\nu(\gamma_0),\dots,\nu(\gamma_g)$, the vector field $V$ does not vanish. The absolute grading $\gr(\bx)$ is the homotopy class of an appropriate extension of $V$ to the union of all $\nu(\gamma_i)$, as we briefly explain. Figure~\ref{nbhd}(a) illustrates two transverse vertical sections of the vector field $V$ in $\nu(\gamma_i)$, for some $i\ge 1$ and Figure~\ref{nbhd}(b) illustrates a vertical section of $V$ in $\nu(\gamma_0)$. Now we substitute $V$ in these neighborhoods by the vector fields illustrated in Figure~\ref{defn}. We note that in $\nu(\gamma_0)$, the vector field in Figure~\ref{defn}(b) has a circle of zeros. We modify the vector field in a neighborhood of this circle so that it rotates clockwise on the $xy$-plane. Let $V^\bx$ be the vector field obtained under this procedure. Then we define $\gr(\bx)$ to be the homotopy class of $V^\bx$. For more details of this construction, see~\cite[\S2.1]{hr}.
\begin{figure}[ht]
\centering
    \begin{overpic}[scale=.25]{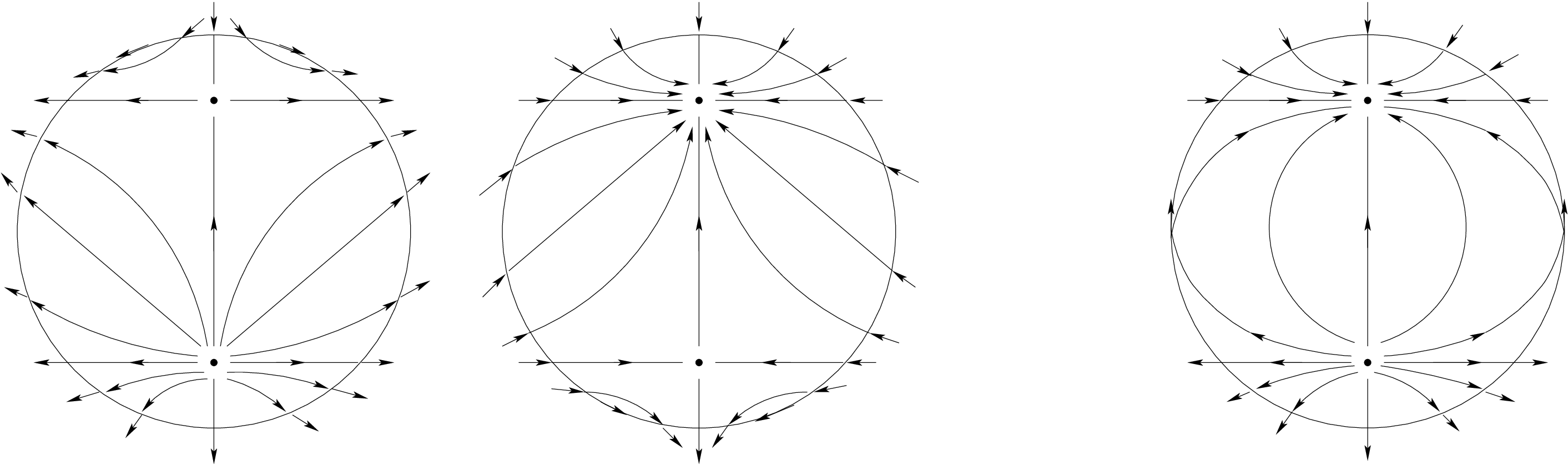}
    \put(10,-3){$xz$-plane}
    \put(40,-3){$yz$-plane}
    \put(27.5,-7){(a)}
    \put(86,-7){(b)}
    \end{overpic}
    \newline
    \newline
    \caption{The vector field $V$}
    \label{nbhd}
\end{figure}
 \begin{figure}[ht]
\centering  
  \begin{overpic}[scale=.25]{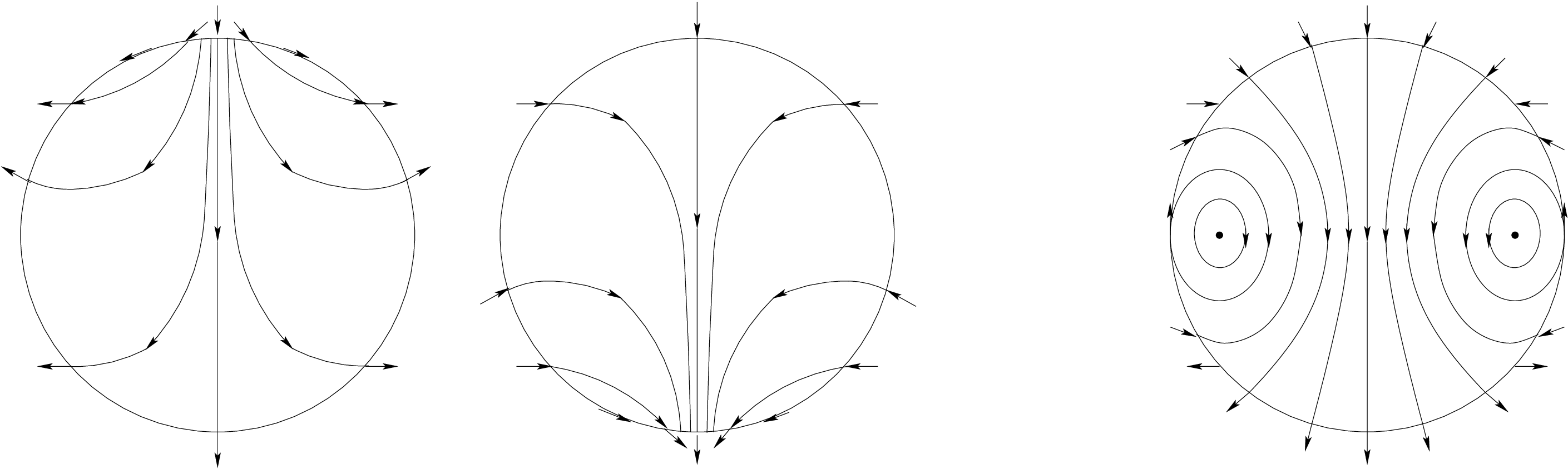}
    \put(10,-3){$xz$-plane}
    \put(40,-3){$yz$-plane}
    \put(27.5,-7){(a)}
    \put(86,-7){(b)}
    \end{overpic}
    \newline
    \newline
    \caption{The modification of $V$}
   \label{defn}
\end{figure}
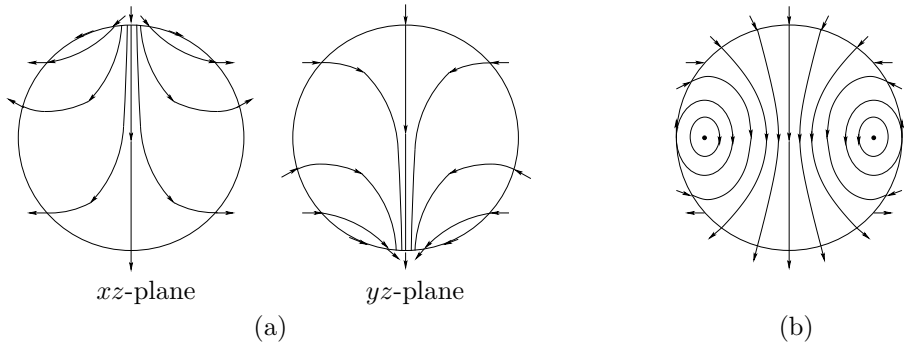

Two generators $\bx$ and $\by$ of $\widehat{CF}(\Sigma,\aa,\bb,z)$ are said to be in the same spin-c structure, if the vector fields $V^\bx$ and $V^\by$ are homotopic in the complement of a 3-ball. For two such generators, one can define a relative grading $\gr(\bx,\by)\in\Z/d$, where $d$ is the divisibility of the first Chern class of the complex line bundle determined by a plane field transverse to $V^\bx$ with the induced orientation. For details, we refer the reader again to~\cite{OS1}. 

We recall that for a given spin-c structure, the space of corresponding homotopy classes of nonvanishing vector fields is an affine space over $\Z/d$, for an appropriate $d$ as above. The main theorem of~\cite{hr} says, in particular, that $\gr$ defines an absolute grading in $\widehat{HF}(Y)$, i.e., if $\bx$ and $\by$ are generators of $\widehat{HF}(Y)$ in the same spin-c structure, then $\gr(\bx,\by)=\gr(\bx)-\gr(\by)\in\Z/d$.

Now, for an intersection point $\bx$ and $n\in\Z$, we define $\gr([\bx,n])=\gr(\bx)+2n$. The inclusion $\iota:\widehat{CF}(\Sigma,\aa,\bb,z)\hookrightarrow CF^+(\Sigma,\aa,\bb,z)$ mapping $x\mapsto[x,0]$ induces the map $\iota_*$ in \eqref{eq:comm}.
It follows from the definition that this map preserves the absolute grading.
 
 \subsection{Heegaard Floer homology and open book decompositions}\label{sec:hfobd}
In this subsection, we will recall how to define the map $\psi'$ and we will show that it preserves the absolute grading.

An {\em open book} is a pair $(S,\varphi)$, where $S$ is a compact oriented surface with boundary and $\varphi$ is a diffeomorphism of $S$ which is the identity on $\partial S$. We will always assume that $\partial S$ is connected. We can construct a topological manifold by considering $S\times[0,1]/\sim$, where $(x,1)\sim(\varphi(x),0)$ for every $x\in S$ and $(x,t)\sim (x,t')$ for all $x\in\partial S$ and $t,t'\in[0,1]$. Given an open book $(S,\varphi)$, let $\bar{S}$ be the surface obtained by gluing an annulus to $S$ and let $\bar{\varphi}$ be a diffeomorphism of $\bar{S}$ obtained by extending $\varphi$ such that $\bar{\varphi}$ is close to the identity in the annulus and equal to the identity in a neighborhood of $\bar{S}$. The quotients obtained by considering $(S,\varphi)$ and $(\bar{S},\bar{\varphi})$ are homeomorphic and $\bar{S}\times[0,1]/\sim$ is actually a smooth manifold. We will say that $(S,\varphi)$ is an {\em open book decomposition} of $Y$ if $Y$ is diffeomorphic to $\bar{S}\times[0,1]/\sim$ where $(\bar{S},\bar{\varphi})$ is constructed from $(S,\varphi)$ as above. The knot $\partial \bar{S}\times\{t\}\subset Y$ is called the {\em binding} and for each $t$ the surface $\bar{S}\times\{t\}\subset Y$ is called a {\em page}.

Let $(S,\varphi)$ be an open book decomposition of $Y$. Up to an isotopy of $\varphi$ relative to $\partial S$, we can assume that in a neighborhood $\nu(\partial S)$ of $\partial S$, we have $\varphi(y,\theta)=(y,\theta-y)$ where we identify $\nu(\partial S)\cong\partial S\times (-\epsilon,0]$. Then $(S,\varphi)$ gives rise to a Heegaard decomposition as follows. The Heegaard surface is $\Sigma:=\bar{S}\times\{1/2\}\cup \bar{S}\times\{0\}$. If we denote the genus of $\bar{S}$ by $g$, then $\Sigma$ has genus $2g$. We choose a set of properly embedded arcs $\ba=\{a_1,\dots,a_{2g}\}$ of $\bar{S}$ such that $\bar{S}\setminus\bigcup_i a_i$ is homeomorphic to a disk. Let $\alpha_i$ be the union of two copies of $a_i$ in $\bar{S}\times\{0\}$ and $\bar{S}\times\{1/2\}$. And let $\beta_i=b_i\cup h(a_i\cap S)$, where $b_i$ is an arc in $\Sigma\setminus(S\times\{0\})$ which is isotopic to $\alpha_i\cap (\Sigma\setminus(S\times\{0\}))$, extends $h(a_i\cap S)$ to a smooth curve in $\Sigma$ and has exactly one intersection with $\alpha_i$ in the interior of $\Sigma\setminus(S\times\{0\})$, see Figure~\ref{fighf}(\subref{hfobd1}). Hence $(\Sigma,\aa,\bb)$ is a Heegaard diagram for $Y$. So $(\Sigma,\bb,\aa)$ is a Heegaard diagram\footnote{This construction is slightly more complicated than that in \cite[\S2.1]{cgh1} and it is not necessary for defining $\widehat{CF}(S,\ba,\varphi(\ba))$, but it will make it easier to choose an appropriate representative of $\gr(\bx)$ in \S\ref{sec:choice}.} for $-Y$.

For each $i$, the circle $\alpha_i$ intersects $\beta_i$ in $\Sigma\setminus(\text{int}(S)\times\{0\})$ at three points. We label them $y_i, y_i',y_i''$, as in Figure~\ref{fighf}(\subref{hfobd1}). We fix a basepoint $z\in S\times\{1/2\}\subset\Sigma$ away from neighborhoods of $\alpha_i\cap (S\times\{1/2\})$. One defines $\widehat{CF}'(S,\ba,\varphi(\ba))$ to be the subcomplex of $\widehat{CF}(\Sigma,\bb,\aa,z)$ generated by $2g$-tuples of intersection points contained in $S\times\{0\}$. One also defines $\widehat{CF}(S,\ba,\varphi(\ba))$ to be the quotient $\widehat{CF}'(S,\ba,\varphi(\ba))/\sim$, where two $2g$-tuples of intersection points in $S\times\{0\}$ are equivalent if they differ by substituting $y_i$ by $y_i'$. There is an induced differential on $\widehat{CF}(S,\ba,\varphi(\ba))$ and the inclusion map induces a map $\widehat{CF}(S,\ba,\varphi(\ba))\to \widehat{CF}(\Sigma,\bb,\aa,z)$ which is an isomorphism in homology by~\cite[Theorem 4.9.4]{cgh1}. The absolute grading on $\widehat{CF}(\Sigma,\bb,\aa,z)$ clearly restricts to an absolute grading of $\widehat{CF}'(S,\ba,\varphi(\ba))$. Let $\bx$ be a generator of $\widehat{CF}'(S,\ba,\varphi(\ba))$ containing $y_i$. Then \[\gr\left(\bx,\bx\setminus\{y_i\}\cup\{y_i'\}\right)=0.\] So absolute grading on the complex $\widehat{CF}(S,\ba,\varphi(\ba))$ is well-defined. Moreover, by definition, the map $\widehat{CF}'(S,\ba,\varphi(\ba))\to\widehat{CF}(\Sigma,\bb,\aa,z)$ preserves the absolute grading. Therefore the isomorphism
\begin{equation}
\psi':\widehat{HF}(\Sigma,\bb,\aa,z)\to \widehat{HF}(S,\ba,\varphi(\ba))\label{eq:2}
\end{equation}
preserves the absolute grading.

\begin{figure}[t]
 \centering
\begin{minipage}[b]{0.4\textwidth}
\def\svgwidth{\textwidth}
 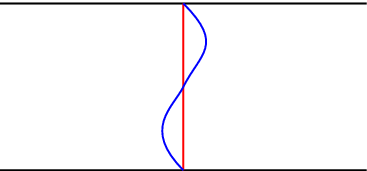
\vspace{0.2cm}
\subcaption{}
\label{hfobd1}
\end{minipage}\qquad\qquad
\begin{minipage}[b]{0.4\textwidth}
\def\svgwidth{\textwidth}
 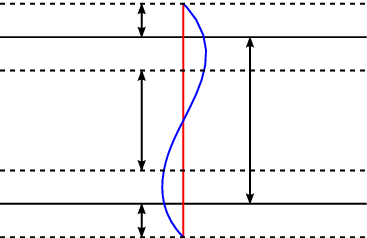
 \vspace{0.1cm}
\subcaption{}
\label{hfobd2}
\end{minipage}
\caption{A neighborhood of the arcs $a_i$}\label{fighf}
\end{figure}

\subsection{The grading on embedded contact homology}\label{abs:ech}
We will now recall the definition of the ECH chain complex and its absolute grading.
Let $Y$ be a closed, oriented three-manifold, let $\lambda$ be a nondegenerate contact form on $Y$ and let $\xi=\ker(\lambda)$. The ECH chain complex $ECC(Y,\lambda)$ is generated by finite orbit sets $\{(\gamma_i,m_i)\}$, where $\gamma_i$ are distinct single orbits of the Reeb vector field associated to $\lambda$, the numbers $m_i$ are positive integers, and $m_1=1$ whenever $\gamma_i$ is hyperbolic. After some extra choices, one can define a differential on $ECC(Y,\lambda)$ that squares to $0$. Its homology is independent of these choices and even of the contact form and is denoted by $ECH(Y)$. For the details of this construction and the invariance, we refer the reader to \cite{hutcast}.

The absolute grading on ECH is defined as follows. Let $\gamma=\{(\gamma_i,m_i)\}$ be an orbit set. The absolute grading $\gr(\gamma)$ is the homotopy class of the vector field obtained by modifying the Reeb vector field in disjoint neighborhoods of the Reeb orbits $\gamma_i$, as follows. For each $i$, fix a small tubular neighborhood of $\gamma_i$ and choose a braid $\zeta_i$ with $m_i$ strands in that neighborhood. Let $L$ be the union of the braids $\zeta_i$. A trivialization $\tau_i$ of $\xi$ over each $\gamma_i$, induces a framing on each $\zeta_i$. Let $\tau$ denote this framing on $L$. Now, for each component $K$ of $L$, let $N_K$ denote a small neighborhood of $K$ in $Y$. We can choose these neighborhoods so that $N_K$ and $N_{K'}$ do not intersect for different components $K$ and $K'$. The framing on $K$ induces a diffeomorphism $\phi_K:N_K\to S^1\times D^2$ and a trivialization of $TN_K$, identifying $\xi=\{0\}\oplus\R^2$ and $R=(1,0,0)$. Using the previous identifications, one can define a vector field $P$ on $N_K$ as
\begin{align*}
 P:S^1\times D^2&\to \R\oplus\R^2\\
 (t,re^{i\theta})&\mapsto(-\cos(\pi r),\sin(\pi r)e^{-i\theta}).
\end{align*}
One now constructs a vector field on $Y$ by defining it to be given by $P$ in each neighborhood $N_K$ and to equal the Reeb vector field in the complement of the union of the neighborhoods $N_K$. Let $P_\tau (L)$ be the homotopy class of this vector field. Now define
\begin{align}
 \gr(\gamma)=P_\tau(L)-\sum_i w_{\tau_i}(\zeta_i)+CZ_{\tau}^I(\gamma),\label{eq:echind}
\end{align}
Here $w_{\tau_i}(\zeta_i)$ denotes the writhe of $\zeta_i$ with respect to $\tau_i$ and $$CZ_{\tau}^I(\gamma)=\sum_i\sum_{k=1}^{m_i} CZ_{\tau}(\gamma_i^k).$$
One can check that $\gr(\gamma)$ does not depend on the choice of $\tau$ or $L$. In~\cite{ir}, Hutchings proved that $\gr$ is an absolute grading on ECH, i.e.,
that if $\gamma$ and $\sigma$ are orbit sets with $[\gamma]=[\sigma]\in H_1(Y)$ then
$$\gr(\gamma)-\gr(\sigma)=I(\gamma,\sigma)\in\Z/d,$$
for an appropriate $d$ depending on $\ker(\lambda)$ and $[\gamma]$. Here $I(\gamma,\sigma)$ denotes the relative grading on ECH, i.e., the ECH index whose definition we shall not need to use.


\subsection{The module $ECC_{2g}(N,\lambda)$}\label{sec:echobd}
We will now recall the definition of $ECC_{2g}(N,\lambda)$ and explain the absolute grading on it.

Let $(S,\varphi)$ be an open book decomposition of $Y$ and let $\lambda$ be a contact form on $Y$ which is {\em adapted} to $(S,\varphi)$, i.e., the Reeb vector field $R_{\lambda}$ is a positively transverse to the interior of the pages and positively tangent to the binding.  As in \S\ref{sec:hfobd}, we assume that $\varphi$ satisfies $\varphi(y,\theta)=(y,\theta-y)$ in a neighborhood of $\partial S$. It follows from \cite[Lemma 2.1.1]{cgh1} that $\lambda$ and $\varphi$ can be chosen so that $\varphi$ is the return map of the Reeb vector field on $S\times\{0\}$. We recall from our construction in \S\ref{sec:hfobd} that for each $t$ the surface $S\times\{t\}$ is a strict subset of a page. Let $N$ be the mapping torus of $\varphi$. Then we can write $Y=N\cup(S^1\times D^2)$. The torus $\partial N$ is foliated by Reeb orbits. Up to a small isotopy of $\lambda$, we can assume that all the Reeb orbits in the complement of $\partial N$ are nondegenerate and that $\partial N$ is a negative Morse-Bott torus. Folloing \cite{cgh1}, we define $ECC_{2g}(N,\lambda)$ to be the $\mathbb{Z}/2$ vector space generated by orbit sets constructed from Reeb orbits in $\text{int}(N)$ and two fixed orbits $\{e,h\}$ on $\partial N$, and whose total homology class intersects a page exactly $2g$ times. Here $e$ and $h$ play the roles of an elliptic and a hyperbolic orbit, respectively. The construction in \S\ref{abs:ech} still works even though $\lambda$ is degenerate. So we obtain an absolute grading on $ECC_{2g}(N,\lambda)$ taking values on $\Vect(Y)$.

 
\section{The main isomorphism}\label{sec:thm}
In this section, we will prove the main part of Theorem \ref{mainthm}, namely that the map $\widetilde{\Phi}$ preserves the absolute grading.

\subsection{The construction of $\widetilde{\Phi}$}\label{sec:tildephi}
We now recall the construction of the map $\widetilde{\Phi}$ on the chain level $$\widetilde{\Phi}:\widehat{CF}(S,\ba,\varphi(\ba))\to ECC_{2g}(N,\lambda).$$ This map is defined by counting rigid holomorphic curves with an ECH-type index equal to 0. We now review the relevant moduli spaces and this ECH-type index.

Throughout this section we fix an open book decomposition $(S,\varphi)$ of $Y$ satisfying the conditions given in \S\ref{sec:hfobd} and we let $N$ be the mapping torus of $\varphi$. We denote by $g$ the genus of $S$ and we let $\lambda$ be a contact form on $Y$ which is adapted to $(S,\varphi)$. In order to prove that $\widetilde{\Phi}$ is an isomorphism, it is necessary to make a more specific choice of $\lambda$ as it is done in~\cite[\S3]{cgh1}, but this particular choice does not affect the absolute grading by Lemma~\ref{lem:cob}.

Let $\pi:\R\times N\to \R\times \R/\Z$ be the map $(s,x,t)\mapsto(s,t)$ and let $B:=(\R\times S^1)\setminus B^c$, where $B^c=(0,\infty)\times(1/2,1)$. We also round the corners of $B$. Now define $W=\pi^{-1}(B)$ and $\Omega=ds\wedge dt+\omega$, where $\omega$ is a certain area form on $S$. Then $(W,\Omega)$ is a symplectic manifold with boundary. It has a positive end, which is diffeomorphic to $S\times[0,1/2]$ and a negative end, which is diffeomorphic to $N$. The map $\pi$ restricts to a symplectic fibration $\pi_{B}:(W,\Omega)\to(B,ds\wedge dt)$ which admits a symplectic connection whose horizontal space is spanned by $\{\partial/\partial s,\partial/\partial t\}$. Now if we take a copy of $\ba=(a_1,\dots,a_{2g})$ on the fiber $\pi_{B}^{-1}(1,1/2)$ and take its symplectic parallel transport along $\partial B$, we obtain a Lagrangian submanifold of $(W,\Omega)$, which is denoted by $L_{\ba}$. For each $a_i\subset \ba$ we denote by $L_{a_i}$ the corresponding component of $L_{\ba}$.

We will consider $J$-holomorphic maps $u:(\dot{F},j)\to(W,J)$ where $(\dot{F},j)$ is a Riemann surface with boundary and punctures, both in the interior and on the boundary. A puncture $p$ is said to be positive or negative if the $s$-coordinate of $u(x)$ converges to $\infty$ or $-\infty$, respectively, as $x\to p$. Now to each generator $\bx$ of $\widehat{CF}(S,\ba,\varphi(\ba))$ we can associate a subset of $S\times[0,1/2]$ given by the union of $x_i\times[0,1/2]$, for all $x_i\in\bx$. Given $\bx$, an orbit set $\gamma=\{(\gamma_i,m_i)\}$ in $ECC_{2g}(N,\lambda)$ and an admissible almost-complex structure $J$, one defines $\mathcal{M}_{J}(\bx,\gamma)$ to be the moduli space of immersed $J$-holomorphic maps $u:(\dot{F},j)\to(W,J)$ satisfying the following conditions:
\begin{enumerate}
 \item[(a)] $u(\partial \dot{F})\subset L_{\ba}$ and each component of $\partial\dot{F}$ is mapped to a different $L_{a_i}$.
 \item[(b)] The boundary punctures are positive and the interior punctures are negative. 
 \item[(c)] At each boundary puncture, $u$ converges to a different chord $x_i\times[0,1/2]$ and every chord $x_i\times[0,1/2]$ is such an end of $u$.
 \item[(d)] At an interior puncture, $u$ converges to an orbit $\gamma_i$ with some multiplicity. For each $i$, the total multiplicity of all ends converging to $\gamma_i$ is $m_i$.
 \item[(e)] The energy of $u$ is bounded.
\end{enumerate}

Let $\overline{W}$ denote the compactification of $W\subset\R\times N$ obtained by compactifying $\R$ to $\R\cup\{-\infty,\infty\}$. A continuous map $u:\dot{F}\to W$ satifying (a)--(d) above can be compatified to a map $\bar{u}:\overline{F}\to \overline{W}$ mapping $\partial \overline{F}$ to \[L_{\bx,\gamma}:=L_{\ba}\cup\left(\{\infty\}\times\bx\times[0,1/2]\right)\cup\left(\{-\infty\}\times
\gamma\right).\] Two such maps $u$, $v$ are said to be homologous if the images of $\bar{u}$ and $\bar{v}$ are homologous in $H_2(\overline{W},L_{\bx,\gamma})$. Let $H_2(W,\bx,\gamma)$ denote the set of homology classes of such maps $u:\dot{F}\to W$.

For a homology class $A\in H_2(W,\bx,\gamma)$, one defines its ECH-index $I(A)$ as follows. Let $u:\dot{F}\to W$ be a continuous map satifying (a)--(d) above such that $[u]=A$ and let $\bar{u}:\overline{F}\to\overline{W}$ be its compactification. Now note that one can view $TS$ as a sub-bundle of $T\overline{W}$. We choose an orientation of the arcs $a_i$, which gives rise to a nonvanishing vector field along each $a_i$. This vector field induces a trivialization $\tau$ of $TS$ along $L_{\ba}\subset \overline{W}$. We extend this trivialization arbitrarily along $\{\infty\}\times\bx\times[0,1/2]$ and along $\{-\infty\}\times\gamma$. Let $c_{\tau}(A)$ denote the first Chern class of $\bar{u}^*TS$ relative to $\tau$. Now let $C_1$ and $C_2$ be distinct embedded surfaces in $\overline{W}$ given by pushing $\bar{u}(\overline{F})$ off along vectors field which are transverse to it and trivial with respect to $\tau$ along the boundary. For more details see~\cite[\S4]{cgh1}. Then $Q_{\tau}(A)$ is defined to be the signed count of intersections of $C_1$ and $C_2$. Now let $\mathcal{L}_0$ be a real, rank one subbundle of $TS$ along $\bx\times[0,1/2]$ defined as follows. At $\bx\times\{0\}$, let $\mathcal{L}_0=T\varphi(\ba)$ and at $\bx\times\{1/2\}$, let $\mathcal{L}_0=T\ba$ in $TS$. Then $\mathcal{L}_0$ is defined by rotating counterclockwise by the minimum possible amount as we travel along $\bx\times[0,1/2]$. One defines $\mu_\tau(\bx)$ to be the sum of the Maslov indices of $\mathcal{L}_0$ along each $x_i\times[0,1/2]$ with respect to $\tau$. The ECH-index is defined as
\begin{equation*}
I(A)=c_{\tau}(A)+Q_\tau(A)+\mu_\tau(\bx)-CZ_{\tau}^I(\gamma)-2g.
\end{equation*}

Now $\widetilde{\Phi}(\bx)$ is defined as follows. The coefficient of an orbit set $\gamma$ in $\widetilde{\Phi}(\bx)$ is the count of maps $u$ in $\mathcal{M}_{J}(\bx,\gamma)$ with $I([u])=0$. As explained in~\cite{cgh1}, for a generic $J$ this count is finite and all the maps that are counted are embeddings.

\subsection{The choice of an appropriate representative of $\gr(\bx)$}\label{sec:choice}

Let $\bx$ be a generator of $\widehat{CF}'(S,\ba,\varphi(\ba))$. We will now explain how to choose a vector field in the equivalence class $\gr(\bx)$ that coincides with the Reeb vector field of a contact form on $Y$ in the complement of a small set in preparation for Proposition \ref{prop:indiso}.

Let $(S,\varphi)$ be an open book decomposition of $Y$ and let $\lambda$ be a contact form on $Y$ which is {\em adapted} to $(S,\varphi)$ satisfying the conditions of \S\ref{sec:echobd}. For each $i=1,\dots,2g$, let $A_i^1$ be a small closed neighborhood of $\alpha_i$ in $\bar{S}$ and let $A_i^2\supset A_i^1$ be a small thickening of it in $\bar{S}$, as in Figure~\ref{fighf}(\subref{hfobd2}). The open book decomposition $(S,\varphi)$ gives rise to a Heegaard diagram $(\Sigma,\bb,\aa,z)$ as in \S\ref{sec:hfobd}. Let $(f,V)$ be a pair which is compatible with $(\Sigma,\bb,\aa,z)$. In what follows when we take the cartesian product of a subset of $\bar{S}$ and an interval in $\R$, we will always take the quotient by the equivalence relation generated by $(x,1)\sim (\bar{\varphi}(x),0)$ for all $x\in \bar{S}$ and $(x,t)\sim (x,t')$ for all $x\in\partial \bar{S}$. So we can see these products as subsets of $Y$. We can assume, without loss of generality, that:
\begin{itemize}
\item The critical points of $f$ belong to $(\bar{S}\times\{1/4\})\cup (\bar{S}\times\{3/4\})$.
\item Every flow line corresponding to a point $y_i''$ as in Figure \ref{fighf}(\subref{hfobd1}) belongs to $A_i^2\times[1/4,3/4]$ and along this flow line $V=-R_\lambda$. Moreover $V$ is not a positive multiple of $-R_\lambda$ elsewhere in $A_i^2\times[1/4,3/4]$.
\item The flow line $\gamma_0$ corresponding to $z$ belongs to $(\bar{S}\setminus\bigcup_iA_i^2) \times[1/4,3/4]$.
\end{itemize}

For $j=1,2$ we let
\[M^j=\nu\left(\left(\bar{S}\setminus\bigcup_{i=1}^{2g}A_i^j\right)\times[1/4,3/4]
\right)\subset Y.\]
Here $\nu(\cdot)$ denotes a small neighborhood in $Y$. We observe that $M^1$ and $M^2$ are 3-balls and $M^1\supset M^2$. See Figure \ref{fig:gamma} for a picture of $M^1$ and $M^2$ in a neighborhood of $a_i\times\{1/2\}$. Let $\bx=(x_1,\dots,x_{2g})$ be a generator of $\widehat{CF}(\Sigma,\bb,\aa,z)$. It follows from the conditions on $(f,V)$ above that we can choose small enough neighborhoods $\nu(\gamma_i)$ for $i=1,\dots,2g$ as in \S\ref{sec:absgrad} which do not intersect $M^1$. We can also assume $\nu(\gamma_0)=M^1$ since $M^1$ contains no index one or two critical points. Under this identification, we require $(\bar{S}\times\{1/2\})\cap M^1$ to be contained in the $xy$-plane. We can now modify $V|_{M^1}$ and define $V^\bx|_{M^1}$ as in \S\ref{sec:absgrad}. Recall that $V^\bx$ is transverse to the $xy$-plane except on a circle which we denote by $\Gamma$, see Figure \ref{defn}(b). Up to a homotopy, we can assume that
\[\Gamma\cap M^2=\partial\bar{S}\cap M^2.\]
Figure \ref{fig:gamma} shows $\Gamma$ in a neighborhood of $a_i\times\{1/2\}$.
So $V^{\bx}$ is positively tangent to the binding in $M^2$ and $V^\bx$ is transverse to the interior of the pages in $M^2$.

\begin{figure}[t]
 \centering
\def\svgwidth{0.4\textwidth}
 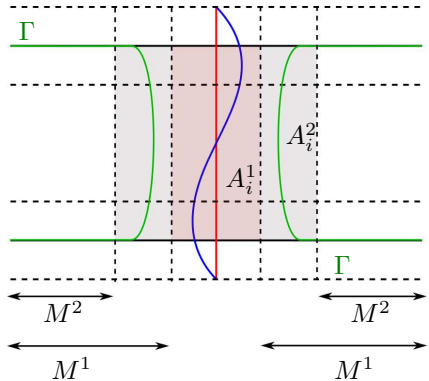
\caption{The curve $\Gamma$}\label{fig:gamma}
\end{figure}

We now let \[Y^0=Y\setminus \nu\left(\bar{S}\times[1/4,3/4]\right).\]
The neighborhood above is chosen to be sufficiently small so that the complement of $Y^0\cup M^2$ is a neighborhood of $\bigcup_i A_i^2\times[1/4,3/4]$.
In $Y^0$ the vector field $V$ is nonvanishing and positively transverse to the pages. So using the construction of the paragraph above, we can assume that $V^{\bx}$ equals the Reeb vector field $R_{\lambda}$ in $Y^0\cup M^2$ except in the neighborhoods $\nu(\gamma_i)$ for $i=1,\dots,2g$. We identify each $\nu(\gamma_i)$ with a subset of $\R^2$ as in Figure \ref{nbhd}(a). We can also assume that $V^{\bx}=-R_\lambda$ along the $z$-axis in $\nu(\gamma_i)$ and that $V^{\bx}\neq -R_\lambda$ in the complement of the $z$-axis in $\nu(\gamma_i)$ for $i=1,\dots,2g$, c.f. Figure \ref{defn}(a).

\subsection{The map $\widetilde{\Phi}$ preserves the grading}

We will now prove a proposition, which is the main ingredient of the proof of Theorem \ref{mainthm}.

\begin{proposition}\label{prop:indiso}
 Let $A\in H_2(W,\bx,\gamma)$, where $\bx$ and $\gamma$ are generators of $\widehat{CF}(S,\ba,\varphi(\ba))$ and $ECC_{2g}(N,\lambda)$, respectively. Then
 \begin{equation}
  \gr(\bx)-\gr(\gamma)=I(A).\label{eq:prop}
 \end{equation}
\end{proposition}

We first recall a relative version of the Pontryagin-Thom construction. Let $v$ and $w$ be nonvanishing vector fields on a closed and oriented three-manifold $Y$. Assume that $v$ and $w$ coincide in $Y\setminus U$, where $U$ is an open set in $Y$. Let $\tau$ be a trivialization of $TY|_U$ and let $p\in S^2$ be a regular value of both $v$ and $w$ seen as maps $U\to S^2$. The one-manifolds $L_v:=v^{-1}(p)$ and $L_w:=w^{-1}(p)$ inherit framings by considering the isomorphisms of their normal bundles with $T_pS^2$ given by $v_*$ and $w_*$ along $L_v$ and $L_w$, respectively. Now if $L_v$ and $L_w$ are contained in the interior of $U$ and are homologous in $U$, there is a link cobordism $C\subset U\times[0,1]$ from $L_v$ to $L_w$. That is, $C$ is a surface such that $\partial C=(L_v\times\{1\})\cup (-L_w\times\{0\})$. The framings on $L_v$ and $L_w$ induce a framing on $C$ along $\partial C$ which we denote by $\nu$. The following lemma is a consequence of the classical Pontryagin-Thom construction and \cite[Lemma 2.3]{hr}.
\begin{lemma}\label{lem:pt}
Let $v$ and $w$ be nonvanishing vector fields and $L_v$ and $L_w$ the links as above. Let $C$ be an immersed cobordism from $L_v$ to $L_w$ and let $\delta(C)$ denote the number of self-intersections of $C$. Let $\tau$ denote the framing on $C$ along $\partial C$ which is induced by the framings on $L_v$ and $L_w$. Then
$$[v]-[w]=c_1(NC,\tau)+2\delta(C).$$
\end{lemma}

\begin{proof}[Proof of Proposition \ref{prop:indiso}]
In order to make it easier to visualize the construction below, we can apply a diffeomorphism of the base $\R/\Z$ of the fibration $\pi$ in \S\ref{sec:tildephi} and change $W$ apropriately so that $\pi(\partial W)=[\epsilon,1-\epsilon]\subset \R/\Z$, for small $\epsilon>0$. This is equivalent to substituting $B^c$ by $(0,\infty)\times(\epsilon,1-\epsilon)$ in \S\ref{sec:tildephi}. 

Let $u:\dot{F}\to W$ be an immersion such that $[u]=A$ and let $\bar{u}:\overline{F}\to\overline{W}$ denote its continuous compactification. We note that by rounding the corners of $\overline{W}$, we obtain a trivial cobordism from $N$ to itself which we denote by $[0,1]\times N$. In particular, \[L_{\ba}\subset \{1\}\times S\times [\epsilon,1-\epsilon]\subset \{1\}\times N.\] We now consider the intersection of the smoothing of $\bar{u}$ with $[\delta,1]\times N$, for small $\delta>0$. We obtain an immersed surface $C$ whose boundary is the union of $\{1\}\times \left(\bx\times[-\epsilon,\epsilon]\right)$, a curve on $L_\ba$ which is transverse to $\{1\}\times S\times\{t\}$ for every $t$ and a link in $\{\delta\}\times N$ which is the union of braids about the Reeb orbits $\{\delta\}\times\gamma_i$ with $m_i$ strands. Here $\gamma=\{(\gamma_i,m_i)\}$.
Let $L\subset N$ be the union of these braids under the identification $N\cong\{\delta\}\times N$.

Let $M^2\subset Y$ and $V^{\bx}$ be as in \S\ref{sec:choice} and let $\widetilde{N}$ be a small open neighborhood of $N\cup (Y\setminus M^2)$. We choose a trivialization of $T\widetilde{N}$ as follows. We first orient the arcs $a_i\subset S$ so as to obtain a nonvanishing vector field on $S$ which is tangent to $a_i$. We can then extend this vector field to all of $S$. We choose a second vector field on $S$ such that these two vector fields form an oriented global frame of $S$. This frame induces a trivialization of $TS$ which gives rise to a trivialization of the pullback bundle of $TS$ over $[\epsilon,1-\epsilon]\times S$. We extend it arbitrarily to a trivialization of the pullback bundle of $TS$ to all of $N$ and we denote this trivialization by $\tau$. Finally we let $R_\lambda$ be the third vector field on $N$ obtaining thus a global frame of $N$. We now extend this frame to a global frame of $\widetilde{N}$ so that $R_\lambda$ is always the third vector field. This frame gives rise to a trivialization $T\widetilde{N}\cong \widetilde{N}\times\R^3$. Note that under this trivialization $R_\lambda$ is the constant vector field $(0,0,1)$.

Let $V_{\tau}^{L}$ be the vector field defined in \S\ref{abs:ech} whose homotopy class is $P_{\tau}(L)$. Then it follows from \S\ref{sec:choice} that we can assume that $V^{\bx}$ and $V_{\tau}^{L}$ coincide in $Y\setminus\widetilde{N}$.
We shall use Lemma~\ref{lem:pt}. We observe that $(V_{\tau}^{L})^{-1}(0,0,-1)=L$. The framing can be calculated by considering the preimage of a vector near $(0,0,-1)$. The corresponding link gives a framing of the normal bundle $NL\cong \xi|_L$ which coincides with $\tau$.

Let $L^{\bx}=(V^{\bx})^{-1}(0,0,-1)$. It follows from the construction of $V^\bx$ in \S\ref{sec:choice} that $L^\bx$ is a slight perturbation of \[\bigcup_i \gamma_{x_i}\cup\bigcup_i \gamma_{y_i''}.\] Here $\gamma_{x_i}$ denotes the flow line of $V$ going through $x_i\in\Sigma$. We note that $L^\bx$ is transverse to the pages. So $\tau$ induces a framing of $L^\bx$ as a link in $N$. Moreover $C\cap \{1\}\times N$ seen as a link in $N$ is isotopic to $L^\bx$ through links that are transverse to the pages. Therefore $(C\cap \{1\}\times N,\tau)$ is framed isotopic to $(L^\bx,\tau)$. By composing $C$ with this framed isotopy, we obtain an immersed cobordism $\widetilde{C}$ between $L$ and $L^\bx$.

Now let $\tau'$ denote the framings on $L^\bx$ and $L$ induced from the Pontryagin-Thom construction. It follows from Lemma \ref{lem:pt} that
\begin{equation}
\gr(\bx)-P_\tau(L)=c_1(N\widetilde{C},\tau')+2\delta(\widetilde{C}).\label{eq:dif1}
\end{equation}

We claim that
\begin{equation}
 c_1(N\widetilde{C},\tau')=c_1(N\widetilde{C},\tau)+\mu_{\tau}(\bx)-2g.\label{eq:mu}
\end{equation}
To prove the claim, we will compute the difference $c_1(N\widetilde{C},\tau')-c_1(N\widetilde{C},\tau)$. This difference is given by $\tau'|_{L^\bx}-\tau|_{L^\bx}$, since $\tau'|_L=\tau|_L$. We orient $L^\bx$ so that it intersects the pages positively, i.e., the orientation follows the flow of $V$ along $\gamma_{x_i}$ and of $-V$ along $\gamma_{y_i''}$. Under the trivialization $(\tau,R_\lambda)$, we have $L^\bx=(V^{\bx})^{-1}(0,0,-1)$.
The framing $\tau'|_{L^\bx}$ is determined by the projection of the vector field $\frac{d}{d\epsilon}|_{\epsilon=0}(V^\bx)^{-1}(\epsilon,0,-1)$ to $TS|_{L^\bx}$. Let $v_{\tau'}$ denote this projection. Let $v_{\tau}$ be the constant vector field $(1,0,0)$ of $T\tilde{N}$. So $v_\tau$ is tangent to $L_{\ba}$ along $\bigcup_i \gamma_{y_i''}$.  So the difference $\tau'|_{L^\bx}-\tau|_{L^\bx}$ is the signed count of turns of $v_{\tau'}$ with respect to $v_\tau$ as we travel along $L^\bx$. We observe that since $\epsilon$ is small, we can assume that $L_{\ba}$ is tangent to the unstable surfaces corresponding to each $\beta_i$ near $\{\epsilon\}\times S$ and to the stable surfaces corresponding to each $\alpha_i$ near $\{1-\epsilon\}\times S$. So the vector field $v_{\tau}$ rotates a quarter of a turn positively about $\gamma_{y_i''}$ for each $i$ with respect to a reference frame in which the stable and unstable surfaces of $V$ are contained in the coordinate axes of $\R^2$, c.f. \cite[\S2]{hr}. Hence $v_{\tau'}$ rotates a quarter of a turn negatively about $\gamma_{y_i''}$ with respect to the same reference frame. So along each $\gamma_{y_i''}$ we obtain a contribution of $-1/2$ to $\tau'|_{L^\bx}-\tau|_{L_\bx}$. Now we compute the difference $\tau'|_{L^\bx}-\tau|_{L^\bx}$ along each $\gamma_{x_i}$. If $v_{\tau}$ makes $1/2+n$ positive half-turns about $\gamma_{x_i}$, we obtain a contribution of $-1/2-n$ to $\tau'|_{L^\bx}-\tau|_{L^\bx}$. In that case, this component will contribute by $-n$ to $\mu_{\tau}(\bx)$. Since there are $2g$ segments $\gamma_{x_i}$ and $\gamma_{y_i''}$, the total difference $\tau'|_{L^\bx}-\tau|_{L^\bx}$ is $\mu_{\tau}(\bx)-2g$ and we have proven~\eqref{eq:mu}.

It remains to compute $c_1(N\widetilde{C},\tau)$. We first note that $c_1(N\widetilde{C},\tau)=c_1(NC,\tau)$ since $\widetilde{C}$ is obtained from $C$ by adding a trivial framed cobordism. We will now use a classical construction in topology. Consider a generic section of $NC$ which is trivial with respect to $\tau$ along $\partial C$. We move $C$ in the direction of this section and we obtain a surface $C'$ which intersects $C$ tranversely. Then
\begin{equation}c_1(NC,\tau)=C\cdot C'-2\delta(C),\label{eq:NC}\end{equation}                                      
where $C\cdot C'$ denotes the signed count of intersections of $C$ and $C'$. But these surfaces are not necessarily $\tau$-trivial. In fact, the linking number of $\partial C$ and $\partial C'$ is $-\sum_i w_{\tau}(\zeta_i)$ in $\{0\}\times \widetilde{N}$ and $0$ in $\{1\}\times \widetilde{N}$. Following a standard calculation in ECH, see e.g.~\cite[\S2.7]{ir}, we obtain \begin{equation}C\cdot C'=Q_{\tau}(A)+\sum_i w_{\tau}(\zeta_i).\label{eq:Q}\end{equation}

Combining~\eqref{eq:echind}, \eqref{eq:dif1}, \eqref{eq:mu}, \eqref{eq:NC} and \eqref{eq:Q}, we obtain~\eqref{eq:prop}.
\end{proof}

Now if $\gamma$ is a term in $\widetilde{\Phi}(\bx)$, it follows from Proposition \ref{prop:indiso} that $\gr(\bx)-\gr(\gamma)=0$. So $\widetilde{\Phi}$ preserves the grading on the chain level, and therefore the isomorphism $\widetilde{\Phi}:\widehat{HF}(S,\ba,\varphi(\ba))\to ECH_{2g}(N,\lambda)$ preserves the grading. 

%
%

\section{ECH and open book decompositions}\label{sec:det}
In this section, we will recall the definition of the map $\psi$ and we will prove that it preserves the absolute grading, which is the last step in the proof of Theorem \ref{mainthm}.

\subsection{The hat version of ECH}


The $U$ map is a degree $-2$ chain map $U:ECC(Y,\lambda)\to ECC(Y,\lambda)$. The chain complex $\widehat{ECC}(Y,\lambda)$ is defined to be the mapping cone of $U$. The homology of $\widehat{ECC}(Y,\lambda)$ is denoted by $\widehat{ECH}(Y,\lambda)$. Again, it follows from~\cite{e1} that the $U$ map in homology does not depend on any choices so we can write $\widehat{ECH}(Y)$. We obtain an exact triangle, as follows.
\begin{eqnarray}
\xymatrix{
ECH(Y)  \ar[rr]^{U} &  &  ECH(Y) \ar[dl] \\ \label{tri:ech}
& \ar[lu]  \widehat{ECH}(Y) &  }.
\end{eqnarray}
We define the absolute grading on $\widehat{ECC}(Y,\lambda,J)$ so that the map $\widehat{ECC}(Y,\lambda)\to ECC(Y,\lambda)$ has degree 0. Hence for $\rho\in\Vect(Y)$, we can write $\widehat{ECH}_{\rho}(Y)$.
We note that the map $ECH(Y,\lambda)\to\widehat{ECH}(Y,\lambda)$ has degree 1.

\subsection{Cobordism maps in ECH}\label{sec:cob}
In this subsection, we will show that the cobordisms maps in ECH defined by Hutchings and Taubes in~\cite{cc2} preserve the absolute grading. This fact will be used in the next subsection.

Let $\lambda$ be a contact form on $Y$. The symplectic action of an orbit set $\gamma=\{(m_i,\gamma_i)\}$ is defined to be $\mathcal{A}_{\lambda}(\gamma):=\sum_{i}m_i\int_{\gamma_i}\lambda$. For $L>0$, the filtered ECH chain complex $ECC^L(Y,\lambda)$ is defined to be the subcomplex of $ECC(Y,\lambda)$ generated by all orbit sets $\gamma$ with $\mathcal{A}_{\lambda}(\gamma)<L$. Since the differential decreases the action, the subgroup $ECC^L(Y,\lambda)$ is indeed a subcomplex. Its homology is denoted by $ECH^L(Y,\lambda)$ and it is independent of the almost-complex structure by~\cite[Theorem 1.3(a)]{cc2}.

For $i=1,2$, let $(Y_i,\lambda_i)$ be a 3-manifold with contact form $\lambda_i$. An exact symplectic cobordism from $(Y_1,\lambda_1)$ to $(Y_2,\lambda_2)$ is a pair $(W,d\lambda)$, where $W$ is a compact 4-manifold, $d\lambda$ is a symplectic form, $\partial W=Y_1\cup(-Y_2)$ and $\lambda|_{Y_i}=\lambda_i$ for $i=1,2$. According to~\cite[Theorem 1.9]{cc2}, such corbordisms induce maps
\begin{equation*}
 \Phi^L(X,\lambda):ECH^L(Y_1,\lambda_1)\to ECH^L(Y_2,\lambda_2).
\end{equation*}
The maps $\Phi^L$ are constructed by taking the composition of the corresponding map in Seiberg-Witten Floer homology and the isomorphism from ECH to Seiberg-Witten Floer homology.

\begin{lemma}\label{lem:cob}
 Let $([0,1]\times Y,d\lambda)$ be an exact cobordism from $(Y,\lambda_1)$ to $(Y,\lambda_0)$. Then, for every $L>0$, the map $\Phi^L([0,1]\times Y,\lambda)$ preserves the absolute grading, i.e. $\Phi^L([0,1]\times Y,\lambda)$ maps $ECH^L_\rho(Y_1,\lambda_1)$ to $ECH^L_\rho(Y_0,\lambda_0)$ for every $\rho\in\Vect(Y)$.
\end{lemma}
\begin{proof}
The maps $\Phi^L([0,1]\times Y,\lambda)$ are defined as a composition of maps
\begin{equation}
ECH^L(Y_1,\lambda_1)\to\widehat{HM}_L(Y,\lambda_1)\to\widehat{HM}_L(Y,\lambda_0)\to ECH^L(Y,\lambda_0).\label{eq:compiso}
\end{equation}
Here $\widehat{HM}_L(Y,\lambda_1)$ and $\widehat{HM}_L(Y,\lambda_0)$ are appropriate filtered Seiberg-Witten Floer cohomology groups, as explained in~\cite{cc2}. The second map in \eqref{eq:compiso} is a filtered version of the cobordism maps defined in~\cite[\S25]{km}. Now it follows from the definition of these maps that if an element of $\widehat{HM}_L(Y,\lambda_1)$ has grading $\rho_1=[v_1]\in\Vect(Y)$, then its image in $\widehat{HM}_L(Y,\lambda_2)$ is the sum of elements of (possibly different) gradings $\rho_0=[v_0]$ such that for each such $\rho_0$ there exists an almost-complex structure $J$ on $[0,1]\times Y$ satisfying $$v_i^\perp=T(\{i\}\times Y)\cap J(T(\{i\}\times Y)),\qquad i=0,1.$$
Now, for $t\in[0,1]$, we let $\xi_t=T(\{t\}\times Y)\cap J(T(\{t\}\times Y))$. Since $T(\{t\}\times Y)$ cannot be invariant under $J$, it follows that $\xi_t$ is a 2-plane field for every $t$. Therefore $\{\xi_t\}$ is a homotopy between $v_0^\perp$ and $v_1^\perp$. Hence $\rho_0=\rho_1$. So the second map in \eqref{eq:compiso} preserves the absolute grading.

Now, the first and third maps in \eqref{eq:compiso} preserve the grading by \cite{dan}. Therefore $\Phi^L([0,1]\times Y,\lambda)$ preserves the grading.
\end{proof}

\subsection{The map $\psi$}\label{sec:psi1}
Let $\lambda$ be a contact form adapted to the open book $(S,\varphi)$ with the compatibility conditions required in \S\ref{sec:echobd}. The map $\psi:ECH_{2g}(N,\lambda)\to \widehat{ECH}(Y)$ is defined to be the composition $\psi=\widehat{\Psi}_1\circ \widehat{\Psi}_2$ as follows.
\begin{equation*}
ECH_{2g}(N,\lambda)\xrightarrow{\quad\widehat{\Psi}_2\quad}  \widehat{ECH}(N,\partial N,\lambda) \xrightarrow{\quad\widehat{\Psi}_1\quad}  \widehat{ECH}(Y). \end{equation*}
We will now recall the definition of $\widehat{ECH}(N,\partial N,\lambda)$, show how to extend the absolute grading to it and prove that $\widehat{\Psi}_1$ and $\widehat{\Psi}_2$ preserve the grading.

We let $ECC(N,\lambda)$ denote the chain complex generated by orbit sets contructed from Reeb orbits in the interior of $N$ and the orbits $\{e,h\}$ where $e$ and $h$ are seen are elliptic and hyperbolic orbits, respectively. The differential counts Morse-Bott buildings of ECH index 1, as explained in \cite{cgh0}. Then $ECC_{2g}(N,\lambda)$ is a subcomplex of $ECC(N,\lambda)$ and the construction of \S\ref{abs:ech} endows $ECC(N,\lambda)$ with an absolute grading taking values in $\Vect(Y)$. Let $ECH(N,\lambda)$ denote its homology. The inclusion induces a map $\iota_*:ECH_{2g}(N,\lambda)\to ECH(N,\lambda)$. Following the notation in \cite{cgh0}, we define $\widehat{ECH}(N,\partial N,\lambda)$ to be the quotient of $ECH(N,\lambda)$ by the equivalence relation generated by $[\gamma]\sim [e\gamma]$ where $\gamma=\prod_i \gamma_i^{m_i}$ is written multiplicatively. The map $\widehat{\Psi}_2$ is the composition of $\iota_*$ with the quotient map, which can be shown to be an isomorphism. It follows from Lemma \ref{lem:greh} below that $\gr(\gamma)=\gr(e\gamma)$. So the absolute grading on $ECH(N,\lambda)$ descends to the quotient $\widehat{ECH}(N,\partial N,\lambda)$. Therefore the map $\widehat{\Psi}_2$ preserves the grading.

The definition of $\widehat{\Psi}_1$ is much more complicated and the proof that it preserves the grading will be the goal of the rest of this paper. Let $ECC^\flat(N,\lambda)$ be the chain complex generated by orbit sets contructed from Reeb orbits in the interior of $N$ and $\{e\}$ and let $ECH^\flat(N,\lambda)$ denote its the homology. Now let $ECH(N,\partial N,\lambda)$ denote the quotient of $ECH^\flat(N,\lambda)$ by the equivalence relations generated by $[\gamma]\sim [e\gamma]$. Similarly to the paragraph above, the quotient map induces an absolute grading on $ECH(N,\partial N,\lambda)$ taking values on $\Vect(Y)$.

In~\cite{cgh0}, Colin, Ghiggini and Honda also constructed an isomorphism
\[
\Psi_1:ECH(N,\partial N,\lambda)\to ECH(Y).
\]
We will now recall the construction of $\Psi_1$, show that it preserves the grading and explain why this implies that $\widehat{Psi}_1$ also preserves the grading. Recall that $Y=N\cup (S^1\times D^2)$. We writethe solid torus $S^1\times D^2$ as $V\cup (T^2\times[0,1])$ where $V$ is a smaller tubular neighborhood of the binding $S^1\times\{0\}$, which is again a solid torus. Let $\lambda_V$ be a contact form on $V$ which is nondegenerate in the interior of $V$ such that the Reeb vector field of $\lambda_V$ is positively transverse to the interior of the pages and positively tangent to the binding and such that $\partial V$ is a {\em positive} Morse-Bott torus. The precise construction of $\lambda_V$ will not be necessary here and we refer the reader to~\cite[\S8.1]{cgh0}. We denote by $e'$ and $h'$ the elliptic and hyperbolic orbits obtained after a Morse-Bott pertubation near $\partial V$. Let $\{L_k\}$ be an increasing sequence such that $\lim_{k\to\infty} L_k=\infty$. Following~\cite[\S9.3]{cgh0}, we can choose a family of contact forms $\{\lambda_k\}$ on $Y$ which equal $\lambda$ in a neighborhood of $N$ and a positive multiple of $\lambda_V$ in a neighborhood of $V$ such that $\lambda_k$ is a Morse-Bott contact form and all Reeb orbits in $T^2\times[0,1]$ have action larger than $L_k$. So as in~\cite[\S9.2]{cgh0}, we can perturb $\{\lambda_k\}$ to a sequence of contact forms $\{\lambda_k'\}$ satisfying, in particular, the following conditions:
\begin{itemize}
 \item $\lambda_k'$ coincides with $\lambda_k$ outside neighborhoods of the Morse-Bott tori.
\item The Reeb orbits of $\lambda_k$ of action less than $L_k$ are nondegenerate and are either the Reeb orbits of $\lambda$ and $\lambda_V$ in the interior of $N$ and $V$, respectively, or one of the orbits $e$, $h$, $e'$ or $h'$.
\end{itemize}
Hence $ECC^{L_k}(Y,\lambda_k')$ is generated by elements of the form $\gamma_V\cdot\gamma_N$, where $\gamma_V$ is an orbit set contructed from Reeb orbits in the interior of $V$ and $\{e',h'\}$, and $\gamma_N$ is a generator of $ECC(N,\lambda)$. For $L>0$, let $ECC^{\flat,L}_{\le k}(N,\lambda)$ be the subcomplex of $ECC^{\flat}(N,\lambda)$ generated by orbit sets $\gamma$ with action $\int_{\gamma}\lambda<L$ and whose total homology class intersects a page up to $k$ times. Following~\cite[\S9.7]{cgh0}, we can define another increasing sequence $\{L_k'\}$ with $\lim_{k\to\infty} L_k'=\infty$ such that the maps $\sigma_k$ below are well-defined.
\begin{align*}
 \sigma_k:ECC^{\flat,L_k'}_{\le k}(N,\lambda)&\to ECC^{L_k}(Y,\lambda_k')\\
\gamma&\mapsto \sum_{i=0}^{\infty}(e')^i\cdot (\partial_N')^i\gamma.
\end{align*}
Here $\partial_N'\gamma$ is defined by the equation $\partial_N \gamma=\partial_N^\flat\gamma+h\partial_N'\gamma$, where $\partial_N$ and $\partial_N^\flat$ are the differentials in $ECC(N,\lambda)$ and $ECC^\flat(N,\lambda)$, respectively. It follows from~\cite[Lemma 9.7.2]{cgh0} that the maps $\sigma_k$ are chain maps so they induce maps \[\sigma_k:ECH_{\le k}^{\flat,L_k'}(N,\lambda)\to ECH^{L_k}(Y,\lambda_k').\]

Following~\cite[Cor. 3.2.3]{cgh0}, there are chain maps $\Phi_k:ECC^{L_k}(Y,\lambda_k')\to ECC^{L_{k+1}}(Y,\lambda_{k+1}')$ which are given by cobordism maps as in \S\ref{sec:cob}. So we obtain a directed system
\begin{equation}
 \xymatrix{
ECC^{\flat,L_k'}_{\le k}(N,\lambda)\ar[r]^{\sigma_k}\ar[d]^{\iota_k}& ECC^{L_k}(Y,\lambda_k')\ar[d]^{\Phi_k}\\ 
ECC^{\flat,L_{k+1}'}_{\le k+1}(N,\lambda)\ar[r]^{\sigma_k}& ECC^{L_{k+1}}(Y,\lambda_{k+1}')
}\label{eq:iota}
\end{equation}
where $\iota_k$ denotes the inclusion. The maps $\Phi_k$ induce maps in homology with respect to which one can take the direct limit $\lim_{k\to\infty} ECH^{L_k}(Y,\lambda_k')$. There is also a nondegenerate contact form $\lambda_0$ and cobordism maps $ECH^{L_k}(Y,\lambda_0)\to ECH^{L_k}(Y,\lambda_k')$. It is shown in~\cite[Cor. 3.2.3]{cgh0} that the direct limit of these maps is an isomorphism
\begin{equation}
ECH(Y,\lambda_0)\cong \lim_{k\to\infty} ECH^{L_k}(Y,\lambda_k').\label{eq:isodir}
\end{equation}
Now we note that $ECH^{\flat}(N,\lambda)=\lim_{k\to\infty}ECH^{\flat,L_k'}_{\le k}(N,\lambda)$. Therefore the maps $\sigma_k$ give rise to a map
\begin{equation*}
 \bar{\sigma}:ECH^\flat(N,\lambda)\to \lim_{k\to\infty} ECH^{L_k}(Y,\lambda_k')\cong ECH(Y,\lambda_0) .
\end{equation*}
The calculations in~\cite[\S9.7]{cgh0} imply that $\bar{\sigma}([\gamma])=\bar{\sigma}([e\gamma])$. Hence we obtain a map $$\Psi_1:ECH(N,\partial N,\lambda)\to ECH(Y).$$
It is shown in~\cite[Theorem 9.8.3]{cgh0} that $\Psi_1$ is an isomorphism.

We will now prove a useful lemma.
\begin{lemma}\label{lem:greh}
 Let $\gamma$ be an orbit set obtained from the Reeb orbits of $\lambda$ in the interior of $N$, respectively, and the orbits $e$, $h$, $e'$ or $h'$. Then $\gr(\gamma)\in \Vect(Y)$ is well-defined. Moreover,
\begin{equation}
\begin{array}{ll}
 \gr(e\gamma)=\gr(\gamma),&\gr(h\gamma)=\gr(\gamma)+1,\\
\gr(e'\gamma)=\gr(\gamma)+2,&\gr(h'\gamma)=\gr(\gamma)+1. \label{eq:eh}
\end{array}
\end{equation}
\end{lemma}
\begin{proof}
To see that $\gamma$ has a well-defined grading, first note that there exists $k_0$ such that $\gamma\in ECC^{L_k}(Y,\lambda_k')$ for every $k\ge k_0$. So we define $\gr(\gamma)$ using the contact form $\lambda_k'$ for some $k\ge k_0$. It follows from Lemma~\ref{lem:cob} that the maps $\Phi_k$ preserve the grading. So $\gr(\gamma)\in\Vect(Y)$ is well-defined.

To prove~\eqref{eq:eh}, we can restrict to the case when $\gamma$ does not contain $e$, $h$, $e'$ or $h'$. The general case is a straightforward consequence of this case. Let $\tau$ be a trivialization of $\xi$ over $\gamma$ and let $L$ be a link as in \S\ref{abs:ech} so that $\gr(\gamma)=P_{\tau}(L)-w_{\tau}(L)+CZ^I_{\tau}(\gamma)$, where $w_{\tau}(L)$ denotes the sum of the writhes of all components of $L$. Let $x\in\{e,h,e',h'\}$. The tangent bundle of the Morse-Bott torus containing $x$ determines a trivialization of $\xi|_x$ which we denote by $\eta$. Let $\zeta$ be a knot obtained by pushing $x$ in a direction which is transverse to the Morse-Bott torus containing $x$ such that $\zeta$ is in the interior of $Y\setminus N$. Then $w_{\eta}(\zeta)=0$. Now let $D$ be a the disk in $Y\setminus N$ bounding $x$. It follows from~\cite[Lemma 3.4(d)]{ir} that 
$$P_{(\tau,\eta)}(L\cup\zeta)-P_{\tau}(L)=c_1(\xi|_D,\eta)=1.$$
Moreover,
\begin{align*}
 CZ_{\eta}(x)=-1,& \text{ if }x=e,\\
CZ_{\eta}(x)=0,&\text{ if }x=h,h',\\
CZ_{\eta}(x)=1,& \text{ if }x=e'.\\
\end{align*}
Therefore it follows from~\eqref{eq:echind} that~\eqref{eq:eh} holds.
\end{proof}

\begin{proposition}
The isomorphism $\Psi_1:ECH(N,\partial N,\lambda)\to ECH(Y)$ preserves the grading.
\end{proposition}

\begin{proof}
Let $\gamma$ be an orbit set in $ECC^{\flat,L_k'}_{\le k}(N,\lambda)$ for some $k$. Since $\partial_N$ decreases the grading by $1$, it follows that $\gr(h\partial_N'\gamma)=\gr(\gamma)-1$. Now, by Lemma~\ref{lem:greh}, $\gr(\partial_N'\gamma)=\gr(\gamma)-2$. Hence for all $0\ge i\ge k$,
$$\gr((e')^i\cdot (\partial_N')^i\gamma)=\gr(\gamma)-2i+2i=\gr(\gamma).$$
So $\sigma_k$ preserves the grading. Now, it is tautological that the inclusion $\iota_k$ in \eqref{eq:iota} preserves the grading. Moreover, by Lemma~\ref{lem:cob}, the maps $\Phi_k$ and the isomorphism~\eqref{eq:isodir} preserve the grading. Hence after passing to homology and taking the direct limit we conclude that $\bar{\sigma}$, and hence $\Psi_1$, preserve the grading. 
\end{proof}

We now define two chain maps as follows.
\begin{equation}
 \begin{array}{ccccccc}
  \iota:ECC^\flat(N,\lambda)&\longrightarrow& ECC(N,\lambda)& \quad&\pi:ECC(N,\lambda)&\longrightarrow& ECC^{\flat}(N,\lambda)\label{eq:defecc}\\
\gamma&\longmapsto&h\gamma& \quad &\gamma_1+h\gamma_2&\longmapsto&\gamma_1
 \end{array}
\end{equation}
Here $\gamma_1$ and $\gamma_2$ do not contain $h$. These maps descend to homology and to the quotients $\widehat{ECH}(N,\partial N,\lambda)$ and $ECH(N,\partial N,\lambda)$. It follows from~\cite[\S9.9]{cgh0} that these maps fit into an exact triangle
\begin{eqnarray}
\xymatrix{
ECH(N,\partial N,\lambda)  \ar[rr] &  &  ECH(N,\partial N,\lambda) \ar[dl]^{\iota_*} \\ \label{tri:echrel}
& \ar[lu]^{\pi_*}  \widehat{ECH}(N,\partial N,\lambda) &  }
\end{eqnarray}
where the map $ECH(N,\partial N,\lambda)\to ECH(N,\partial N,\lambda) $ is a version of the $U$ map. Moreover there exists an isomorphism $\widehat{\Psi}_1:ECH(N,\partial N,\lambda)\to \widehat{ECH}(Y)$ such that $\Psi_1$ and $\widehat{\Psi}_1$ give an isomorphism from \eqref{tri:echrel} to \eqref{tri:ech}. It follows from \eqref{eq:defecc} that $\iota_*$ increases the grading by $1$ and that $\pi_*$ preserves the grading. Hence we obtain the following diagram.
\begin{equation*}
\xymatrix{
\dots\ar[r]& ECH_{\rho-1}(N,\partial N,\lambda)\ar[r]\ar[d]^{\Psi_1}& \widehat{ECH}_\rho(N,\partial N,\lambda) \ar[r]\ar[d]^{\widehat{\Psi}_1}& ECH_{\rho}(N,\partial N,\lambda)\ar[r]\ar[d]^{\Psi_1}&\dots \\
\dots\ar[r]^-U& ECH_{\rho-1}(Y)\ar[r]& \widehat{ECH}_\rho(Y) \ar[r]& ECH_{\rho}(Y)\ar[r]^-U&\dots
} 
\end{equation*}
Therefore $\widehat{\Psi}_1$ preserves the grading.


\begin{thebibliography}{10}

\bibitem{cgh1}
Vincent Colin, Paolo Ghiggini, and Ko~Honda, \emph{The equivalence of
  {H}eegaard {F}loer homology and embedded contact homology via open book
  decompositions {I}}, 2012, arXiv:1208.1074.

\bibitem{cgh2}
\bysame, \emph{The equivalence of {H}eegaard {F}loer homology and embedded
  contact homology via open book decompositions {II}}, 2012, arXiv:1208.1077.

\bibitem{cgh3}
\bysame, \emph{The equivalence of {H}eegaard {F}loer homology and embedded
  contact homology via open book decompositions {III}: from hat to plus}, 2012,
  arXiv:1208.1526.

\bibitem{cgh0}
\bysame, \emph{Embedded contact homology and open book decompositions}, 2013,
  arXiv:1008.2734.

\bibitem{dan}
Daniel Cristofaro-Gardiner, \emph{The absolute gradings on embedded contact
  homology and {S}eiberg-{W}itten {F}loer cohomology}, Algebr. Geom. Topol.
  \textbf{13} (2013), no.~4, 2239--2260.

\bibitem{hr}
Yang Huang and Vinicius Gripp Barros Ramos, \emph{An absolute grading on
  {H}eegaard {F}loer homology by homotopy classes of oriented 2-plane fields},
  2011, to appear in J. Symplectic Geom., arXiv:1112.0290.
  
\bibitem{hutcast}
Michael Hutchings, \emph{Lecture notes on embedded contact homology}, Contact and symplectic topology, Bolyai Soc. Math. Stud. \textbf{26} (2014), 389--484.


\bibitem{ir}
\bysame, \emph{The embedded contact homology index revisited}, New perspectives
  and challenges in symplectic field theory, CRM Proc. Lecture Notes, vol.~49,
  Amer. Math. Soc., Providence, RI, 2009, pp.~263--297.



\bibitem{cc2}
Michael Hutchings and Clifford H. Taubes, \emph{Proof of the {A}rnold chord conjecture in three dimensions
  {II}\/}, Geom. Topol. \textbf{17} (2013), 2601--2688.

\bibitem{km}
Peter Kronheimer and Tomasz Mrowka, \emph{Monopoles and three-manifolds}, New
  Mathematical Monographs, vol.~10, Cambridge University Press, Cambridge,
  2007.

\bibitem{OS1}
Peter Ozsv{\'a}th and Zolt{\'a}n Szab{\'o}, \emph{Holomorphic disks and
  topological invariants for closed three-manifolds}, Ann. of Math. (2)
  \textbf{159} (2004), no.~3, 1027--1158.

\bibitem{e1}
Clifford H. Taubes, \emph{Embedded contact homology and {S}eiberg-{W}itten
  {F}loer cohomology {I}}, Geom. Topol. \textbf{14} (2010), no.~5, 2497--2581.

\end{thebibliography}

\providecommand{\bysame}{\leavevmode\hbox to3em{\hrulefill}\thinspace}
\providecommand{\MR}{\relax\ifhmode\unskip\space\fi MR }
\providecommand{\MRhref}[2]{%
  \href{http://www.ams.org/mathscinet-getitem?mr=#1}{#2}
}
\providecommand{\href}[2]{#2}

\end{document}